\documentclass[11p,reqno]{amsart}
\usepackage{amsmath}
\usepackage{amsfonts}
\usepackage{amssymb}
\usepackage{graphicx}
\usepackage{color}
\newtheorem{theorem}{Theorem}[section]
\newtheorem*{theorem*}{Theorem}

\newtheorem{corollary}[theorem]{Corollary}
\newtheorem{proposition}{Proposition}[section]

\newtheorem{definition}[theorem]{Definition}

\numberwithin{equation}{section}

\begin{document}
	\title[Universality]{Universality of the divergence}

\author{Lei Ni}

\address{Lei Ni. Zhejiang Normal University, Jinhua, Zhejiang, China, 321004}
\address{department of Mathematics, University of California, San Diego, La Jolla, CA 92093, USA}
\email{lni.math.ucsd@gmail.com}

\author{Yijian Zhang}

\address{Yijian Zhang. School of Mathematical Sciences,  University of Science and Technology of China, Hefei, Anhui, China, 230026}

\email{zyj\_math@mail.ustc.edu.cn}

\subjclass[2010]{}

\begin{abstract} Algebraists asked whether or not  an operator on the module of smooth sections of the  tangent bundle over the commutative ring of smooth functions of a smooth (orientable) manifold (can be any piece of a compact or a complete manifold) can be characterized by two axioms. In this note we confirm this for any smooth manifold $M$ under the assumption that $H^1(M, \mathbb{R})=\{0\}$.
\end{abstract}

\maketitle

\section{Introduction}

The purpose of this paper is to give an axiomatic description of the divergence operator acting on smooth vector fields. We first consider the following algebraic definition. Let $M^n$ ($n=\dim_{\mathbb{R}}(M)$) be a smooth manifold. We denote $\mathcal{A}(M)$ the ring of all smooth functions, and $\mathfrak{X}(M)$ the smooth sections of the tangent bundle $TM$, namely all smooth vector fields, which is viewed as a module over $\mathcal{A}(M)$.

 \begin{definition}\label{def:11}
 A linear operator $D:\mathfrak{X}(M) \to \mathcal{A}(M)$ is called a {\bf{divergence operator}} if it satisfies the following two conditions:

 (i) $D$ is a $1$-cocycle in the sense that for any $X_1, X_2\in \mathfrak{X}(M)$,
 \begin{equation}\label{eq:11}
 D[X_1, X_2]=[D(X_1), X_2]+[X_1, D(X_2)];
 \end{equation}

 (ii) For $f\in\mathcal{A}(M)$ and $X\in \mathfrak{X}(M)$,
 \begin{equation}\label{eq:12}
 D(f \cdot X)=f D(X) + X(f).
 \end{equation}

 \end{definition}

An explanation of the right hand side of (\ref{eq:11}) is in order. First it does not say that $D$ is a derivation since $D$ is not a self-map of the Lie algebra $\mathfrak{g}=\mathfrak{X}(M)$. The right hand side is understood in terms of operators acting on $\mathcal{A}(M)$. Namely, for any $f\in \mathcal{A}(M)$,
\begin{eqnarray*}
&\, &[D(X_1), X_2] (f) = D(X_1) \cdot X_2(f)- X_2(D(X_1)\cdot f)=-X_2(D(X_1)) \cdot f; \\
&\, &[X_1, D(X_2)] (f) = X_1(D(X_2)\cdot f)-D(X_2)\cdot  X_1 (f)=X_1(D(X_2))\cdot f.
\end{eqnarray*}
Equivalently we may write it as
$$
[D(X_1), X_2]+[X_1, D(X_2)] =X_1(D(X_2))-X_2(D(X_1))\in \mathcal{A}(M).
$$
Then the condition (\ref{eq:11}) can be interpreted as  that the operator $D: \mathfrak{X}(M)\to \mathcal{A}(M)$ is a closed 1-form taking values in the linear space $\mathcal{A}(M)$. It represents an element in the cohomology space (of Lie algebra $\mathfrak{g}=\mathfrak{X}(M)$) $H^1(\mathfrak{X}(M),\mathcal{A}(M))$ with respect to the natural linear representation, namely viewing $X$ as a linear map of $\mathcal{A}(M)$ to itself.  One may refer \cite{IT} for more details on the Lie algebra cohomology $H^1(\mathfrak{g}, V)$ associated with general representation $\rho: \mathfrak{g}\to \mathfrak{gl}(V)$ for any vector space $V$.

Recall that for a Riemannian manifold $(M, g)$, the divergence operator $\operatorname{div}$ on a vector field $X$, is defined by taking trace of the linear map $Y\xrightarrow{}\nabla_{Y}X$, where $\nabla$ is the Levi-Civita connection. If $G=\det(g_{ij})$ with respect to a local coordinate chart, and $X=\sum_{i=1}^n X^i \frac{\partial}{\partial x^i}$,
$$
\operatorname{div}(X)=\sum \frac{1}{\sqrt{G}}\frac{\partial}{\partial x^i}\left(\sqrt{G} X^i\right).
$$

However, the divergence operator can be defined without a Riemannian metric, in particular, not involving its Levi-Civita connection. What is needed is a volume form, namely a nowhere-vanishing top form (this is equivalent to orientability of the manifold), and the Lie derivative of smooth vector fields on differential forms.  Please refer to Ch. 5, Section 7 of \cite{Mat} and Section 3.9 of \cite{AF}. Precisely, for a volume form $\Omega$, the divergence is defined by
\begin{equation}\label{eq:13}
L_X \Omega =\operatorname{div}_\Omega (X) \Omega.
\end{equation}
On a Riemannian manifold $(M, g)$, the previously defined divergence operator in terms of Levi-Civita connection coincides with the second one  with respect to the standard
volume form $\sqrt{G}dx^1\wedge\cdots \wedge dx^n$ (cf. Page 71, Lemma 5.12 of \cite{Sakai}). Since a lot of Riemannian metrics can de defined on a given smooth manifold, it follows that there are many corresponding volume forms, hence many divergence operators associated with them.  The most useful property of the operator defined via (\ref{eq:13}) is that for $X$ with compact support,
\begin{equation}\label{eq:14}
\int_M \operatorname{div}_\Omega (X) \Omega=0,
\end{equation}
which follows from the Cartan formula $L_X=d\cdot \iota_X +\iota_X\cdot d $ and the Stokes' theorem.

The main result here is that

\begin{theorem}\label{thm:main}
If $M$ is an orientable manifold with $H^1(M, \mathbb{R})=\{0\}$, then  $D$ satisfies $(\ref{eq:11})$ and $(\ref{eq:12})$ if and only if $D$ is the divergence operator $\operatorname{div}_\Omega$ with respect to some volume form $\Omega$  on $M$. In particular, the result holds on any simply connected  manifold.
\end{theorem}

This gives an axiomatic characterization of the general divergence operators on a smooth manifold with $H^1(M, \mathbb{R})=\{0\}$. The result and the title of the paper is motivated by the axiomatic definition of the determinant in linear algebra (for example see section I.F. of \cite{Ar}). 

\section{Proof of the result}

First it is easy to check that any $\operatorname{div}_\Omega$ satisfies (\ref{eq:11}) and (\ref{eq:12}). The key is one of Cartan identities:
\begin{equation}\label{eq:cartan}
L_{[X, Y]}=[L_X, L_Y].
\end{equation}
We leave the detailed checking to the readers.
Now we focus on the proof that given any $D$, there exists a volume form $\Omega_D$ such that $D=\operatorname{div}_{\Omega_D}$.

Given a volume form $\Omega_0$, let $D_0$ be the divergence operator with respect to $\Omega_0$ defined by (\ref{eq:13}). First let $E=D-D_0$. By (\ref{eq:11}) and (\ref{eq:12}) we have that the operator $E$ satisfies that
\begin{eqnarray}
&\, & E(f \cdot X) = f \cdot E(X); \label{eq:21}\\
&\,& E([X, Y])= [E(X), Y]+[X, E(Y)]. \label{eq:22}
\end{eqnarray}
The equation (\ref{eq:21}) implies that $E$, which is still a $1$-form of the Lie algebra $\mathfrak{X}(M)$ valued in $\mathcal{A}(M)$, is in fact linear over $\mathcal{A}(M)$. Pick a local coordinate chart and write
$E(\frac{\partial}{\partial x^i})=E_i$, the $\mathcal{A}(M)$-linearity makes $E$ a true $1$-form $\in \Omega^1(M)$, which can be written locally as
$$
E=\sum_{i=1}^n E_idx^i
$$ Now use $[\frac{\partial}{\partial x^i}, \frac{\partial}{\partial x^j}]=0$ we have by (\ref{eq:22}) that
$$
0=\frac{\partial E_j}{\partial x^i}-\frac{\partial E_i}{\partial x^j},
$$
hence $E$ is $d$-closed.  By assumption there exists a smooth function $f$ such that $E=df$, namely $E(X)=X(f)$. Locally we have $E_i=\frac{\partial f}{\partial x^i}$. Now since $D=D_0+E$ and $D_0=\operatorname{div}_{\Omega_0}$, writing $\Omega=e^f\Omega_0$, direct calculation shows that for any $X\in \mathfrak{X}(M)$
\begin{eqnarray*}
\operatorname{div}_{\Omega} (X)&=& \frac{L_X (e^f \Omega_0)}{e^f \Omega_0} \\
&=&\frac{1}{e^f} \left( e^f \cdot \operatorname{div}_{\Omega_0} ( X)+X(e^f)\right)\\
&=& \operatorname{div}_{\Omega_0} ( X)+ X(f)\\
&=& (D_0+E)(X).
\end{eqnarray*}
This proved the claim that $D$ is given by $\operatorname{div}_{\Omega}$.

 Notice that, the above argument shows that the general divergence operators is completely characterized by the formula $D=\operatorname{div}_{\Omega_0}+E$ for some closed form $E$, and two $\operatorname{div}_{\Omega}$ differ from an exact form, thus $H^1(M, \mathbb{R})=\{0\}$ is also a necessary condition for Theorem \ref{thm:main}.

 \begin{corollary} Without the assumption $H^1(M, \mathbb{R})=\{0\}$. The space of operators satisfying $(\ref{eq:11})$ and $(\ref{eq:12})$ consists of
 $\{ \operatorname{div}_{\Omega_0}+E\}$ with $E$ being any closed $1$-form.
 \end{corollary}
 \begin{proof}
 Observe that for a $1$-form $E$, and $X_i\in \mathfrak{X}(M)$, $i=1,2 $
 $$
 dE (X_1, X_2)=X_1(E(X_2))-X_2(E(X_1))-E([X_1, X_2]).
 $$
Hence $E$ satisfies (\ref{eq:11}) if and only if $dE=0$. \end{proof}

   A couple of remarks are in order. First, given any volume form, it is easy to find a Riemannian metric such that its volume form is the given one (the harder version is the Calabi conjecture solved by Yau for the K\"ahler setting which demands a metric within a K\"ahler class).

Secondly, for any affine connection $\widetilde{\nabla}$, define $A_X(\cdot)=L_X(\cdot)-\widetilde{\nabla}_X(\cdot)$. Then one can also define the divergence with respect to $\widetilde{\nabla}$ by $\operatorname{div}_{\widetilde{\nabla}}(X)=-\operatorname{trace}(A_X)$. Since when $\widetilde{\nabla}$ has no torsion, this is the same as  the trace of $\widetilde{\nabla}_{(\cdot)}X$, it extends the one for Levi-Civita connection. There exists the following result which also connects $\operatorname{div}_{\widetilde{\nabla}}$ with the one defined via the volume form in (\ref{eq:13}). Our definition of the divergence  for any affine connection $\widetilde{\nabla}$ does not coincide with taking the trace of $\widetilde{\nabla}_{(\cdot)}X$  in some of the literatures. However it has the advantage of (\ref{eq:14}). Moveover it can be viewed as the trace of the associated torsion connection $\widetilde{\nabla}^T$, which is  defined by
$$
\widetilde{\nabla}^T_YX=-A_{X}(Y).
$$

\begin{proposition}[Kobayashi-Nomizu]\label{prop:KN} Let $\widetilde{\nabla}$ be an affine connection on $M$.
Assume that $\Omega$ is a volume form satisfying $\widetilde{\nabla}(\Omega)=0$, which means $\Omega$ is invariant under parallelism with respect to $\widetilde{\nabla}$, then $\operatorname{div}_{\widetilde{\nabla}}=\operatorname{div}_\Omega$. The converse also holds.
\end{proposition}
\begin{proof} See page 282 of \cite{KN}. \end{proof}

Hence in view of that the Riemannian volume form is parallel with respect to the Levi-Civita connection, $\operatorname{div}_{\widetilde{\nabla}}$ defined as above provides a nature generalization of the divergence operator defined for a Riemannian metric and Levi-Civita connection via the trace of $\nabla_{(\cdot)} X$.

Furthermore, for a K\"ahler manifold $(M^m, \omega)$ one may consider $\mathfrak{X}'(M)$ and $\mathfrak{X}''(M)$, namely the vector fields of $(1, 0)$ and $(0, 1)$ types. Since metric is parallel with respect to the Levi-Civita connection, one can use the definitions in Section 4 of \cite{NN} to define the divergence for $X'\in \mathfrak{X}'(M)$ and $X''\in \mathfrak{X}''(M)$. It can be checked that they coincide with $\operatorname{div}_{\omega^m}$ defined via the Lie derivative. In this case $D$ splits into the sum of two maps $D': \mathfrak{X}'(M)\to \mathcal{A}(M)$ and $D'':\mathfrak{X}''(M)\to \mathcal{A}(M)$ satisfying (\ref{eq:11}) and (\ref{eq:12}). For the volume form with an extra  weight, the above consideration of Kobayashi-Nomizu holds, and  the divergence operator finds its uses in proving the reductivity of automorphism group for a compact K\"ahler-Einstein manifolds with positive scalar curvature \cite{Mat1}.

Finally, the orientability is not essential here since one may replace the volume form by the density and define the divergence accordingly via (\ref{eq:13}) (cf. page 30 of \cite{BGV}). Moreover Theorem \ref{thm:main} still holds for a $s$-density ($s\in\mathbb{R}$) and the associated $s$-divergence operator, which satisfies (\ref{eq:11}) and replaces (\ref{eq:12}) by
 \begin{equation}\label{eq:2s}
  D(f \cdot X)=f D(X) + s\cdot X(f).
 \end{equation}
In particular, since $0$-densities are functions, $df$ is a $0$-divergence operator and is the only type satisfying (\ref{eq:11}) and (\ref{eq:2s})(case $s=0$) for manifolds with $H^1=\{0\}$.

\bigskip

\noindent\textbf{Acknowledgments.} {The authors thank Efim Zelmanov for bringing this question to our attention, Peng Du, Vyacheslav Futorny  for their interests.}

\end{document}